\documentclass{cimart}

\usepackage{tikz}
\usetikzlibrary{matrix, decorations.pathreplacing, positioning}

\DeclareMathOperator{\pf}{pf}
\newcommand{\G}{\Gamma}

\title{GKM theory for manifolds of isospectral matrices in Lie type D}

\authors{Evgeny Zhukov}

\authorinfo{HSE University, Russia}{zhukovzh@mail.ru}

\abstract{%
    We study the manifold $Q_{\G, \lambda}$ of isospectral real skew-symmetric matrices with a prescribed sparsity pattern determined by a graph $\G$. The compact torus $T^n$ acts naturally on $Q_{\Gamma,\lambda}$ by conjugation, and this action can be studied using GKM theory. We prove two results about this manifold and its GKM graph. The first theorem describes how the GKM graph of $Q_{\G, \lambda}$ is obtained from the GKM graph of the corresponding manifold $M_{\G, \lambda}$ of isospectral Hermitian matrices. The second theorem shows that equivariant formality of $Q_{\Gamma,\lambda}$ implies equivariant formality of $M_{\Gamma,\lambda}$, whereas the converse fails.
    }



\VOLUME{35}
\YEAR{2027}
\ISSUE{1}
\NUMBER{2}
\DOI{https://doi.org/10.46298/cm.18277}
\editinfo{May 26, 2026}{October 27, 2026}{Eric Swartz}

\begin{document}

\section*{Introduction}
\addcontentsline{toc}{section}{Introduction}
Let $\G$ be a simple connected graph on the vertex set $V_{\Gamma} = [n] = \{1, \ldots, n\}$, where \mbox{$n\geq2$}, with edge set $E_{\Gamma}$. To this graph we associate two manifolds (see Constructions~\ref{def_q} and \ref{def_m}): the manifolds $M_{\Gamma,\lambda}$ and $Q_{\Gamma,\lambda}$ of $\G$-shaped Hermitian and skew-symmetric matrices, respectively, with a fixed generic spectrum $\lambda$. Both manifolds carry natural actions of the compact $n$-dimensional torus $T^n$ by conjugation. We call these spaces \emph{isospectral matrix manifolds}.

Manifolds $M_{\G, \lambda}$ for particular classes of graphs were studied in \cite{A1}, \cite{Arrow}, \cite{Tridiag}. In \cite{A2}, a general approach to $M_{\G, \lambda}$ based on computational topology was proposed. It was shown that these manifolds admit a useful combinatorial description; the associated GKM graph (equivalently, the 1-skeleton of the torus action) plays a key role.

It was noted in \cite{A1} that $M_{\G, \lambda}$ corresponds to semisimple Lie type A, and that the general setting can be extended to other Lie types. In the present paper we study the manifold $Q_{\G, \lambda}$, which corresponds to Lie type D. An explicit description of its GKM graph is given in Proposition~\ref{last}. Using this result together with results from \cite{A1}, we prove Theorem~\ref{Th}.

\begin{theorem}\label{Th}
    The GKM graph of $Q_{\G, \lambda}$ can be obtained from the GKM graph of $M_{\G, \lambda}$ by the sign-vector construction described in Construction~\ref{graph}.
\end{theorem}
The GKM graph of $M_{\G, \lambda}$ is described in Proposition~\ref{Her} of Section~\ref{M_GKM}. Construction~\ref{graph} then gives the precise connection between this graph and the GKM graph of $Q_{\G, \lambda}$.

Equivariant formality is a fundamental hypothesis in GKM theory. Many varieties with algebraic torus actions in complex projective geometry are equivariantly formal by the Białynicki-Birula decomposition \cite{BB}. On the other hand, toric topology provides examples of torus actions that are not equivariantly formal, together with methods for computing their cohomology. It was shown in \cite{Ay1} that manifolds of isospectral Hermitian matrices of certain sparsity types are not equivariantly formal. Thus, it is natural to ask when matrix manifolds such as $Q_{\G, \lambda}$ are equivariantly formal. The most general result for manifolds of Hermitian matrices was proved in \cite{A2}; previously, no corresponding result was known for other natural types of matrices. The present paper partially fills this gap.

\begin{theorem}\label{main}
Equivariant formality of $Q_{\Gamma,\lambda}$ implies equivariant
formality of $M_{\Gamma,\lambda}$. The converse implication does not
hold in general.
\end{theorem}
\par

A graph is a \emph{unit interval graph}, also called an \emph{indifference graph},
if it is the intersection graph of unit intervals on the real line.
For generic spectra, this class characterizes equivariant formality
of $M_{\Gamma,\lambda}$ \cite{A2}. Theorem~\ref{main} therefore makes
this condition necessary for equivariant formality of
$Q_{\Gamma,\lambda}$. It is not sufficient: the path $P_3$ is a unit
interval graph, but $Q_{P_3,\lambda}$ is not equivariantly formal.


The paper is organized as follows. Section \ref{sect2} reviews background on torus actions, equivariant cohomology, and GKM theory. Section \ref{sect3} defines the manifolds $M_{\G, \lambda}$ and $Q_{\G, \lambda}$ and discusses their simplest examples.
Section \ref{Proof} proves Theorem~\ref{Th} by describing the GKM graphs of both manifolds (Propositions~\ref{Her} and \ref{last}). Section \ref{sect5} first compares equivariant formality of the two families, then proves Theorem~\ref{main} by constructing a counterexample to the converse implication.

\section{Toric topology}\label{sect2}
This section reviews definitions and constructions from toric topology and GKM theory used in the paper. We follow the terminology and conventions of \cite{Ay1,A1}.

\subsection{Basics of group actions}
We recall the terminology used below. Let
$\phi\colon G\times X\to X$ be a smooth action. It induces a
homomorphism
\[
\Phi\colon G\longrightarrow \operatorname{Diff}(X),
\qquad \Phi(g)=\phi(g,\mathord{\cdot}),
\]
where $\operatorname{Diff}(X)$ is equipped with the compact-open
topology. The action is \textit{effective} if
$\ker\Phi=\{e\}$. For $x\in X$, its orbit and stabilizer are denoted by
\[
G(x)=\{\phi(g,x)\mid g\in G\},\qquad
G_x=\{g\in G\mid \phi(g,x)=x\}.
\]
The orbit space $X/G$ is endowed with the quotient topology. For the
compact Hausdorff spaces considered below, the quotient by a compact
group is compact and Hausdorff. Since all acting groups in this paper
are tori, no distinction between left and right actions is needed.

\subsection{Tangent representation}
Let a compact torus $T$ of dimension $k$ act on a closed, connected,
oriented smooth manifold $X$ of real dimension $2m$. We assume that
the fixed point set $X^T$ is nonempty and finite. The abelian group
$N=\operatorname{Hom}(T,S^1)\cong\mathbb Z^k$ is called the weight
lattice.

For a disconnected compact $T$-manifold, the constructions and results
below are applied to its individual connected components, which are
preserved by the torus action. Its GKM graph is the disjoint union of
the component graphs, and the action is equivariantly formal if and
only if its restriction to every component is equivariantly formal.

\begin{construction} \label{tang}

If $x\in X^T$ is an isolated fixed point, the action induces the
\textit{tangent representation} of $T$ on $T_xX$. Its weights
$\alpha_{x,1},\ldots,\alpha_{x,m}\in\operatorname{Hom}(T,S^1)$ are
defined by a decomposition
\[
T_xX\cong V(\alpha_{x,1})\oplus\cdots\oplus V(\alpha_{x,m}),
\]
where $V(\alpha)$ is the one-dimensional complex representation
$t z=\alpha(t)z$, regarded as a two-dimensional real representation.
\end{construction}

All weights are nonzero because $x$ is isolated. The underlying real
two-dimensional representation is unchanged when a weight is replaced
by its negative; hence the weights are determined only up to sign.

\subsection{Borel construction}
The Milnor construction provides a space $EG$ with the following
properties:
\begin{enumerate}
    \item $EG$ is contractible;
    \item $G$ acts on $EG$ freely, i.e., $G_{x} = \{e\}$ for all $x \in EG$.
\end{enumerate}
The orbit space $BG=EG/G$ is the \textit{classifying space} of $G$.

If $X$ is a $G$-space, the product $EG\times X$ carries the diagonal
action
\[
g\cdot(a,x)=(ag^{-1},\phi(g,x)).
\]
The orbit space $EG\times_GX=(EG\times X)/G$, also denoted by $X_G$,
is the \textit{Borel construction} or \textit{homotopy quotient}.

\subsection{Equivariant formality}
Unless otherwise stated, cohomology and equivariant formality are
understood with integer coefficients. For a coefficient ring $R$, the
Borel construction gives a Serre fibration
\[
X\longrightarrow X_G\xrightarrow{p}BG
\]
and a spectral sequence
\[
E_2^{p,q}=H^p(BG;H^q(X;R))\Longrightarrow H^{p+q}(X_G;R).
\]
\begin{definition} 
For a topological group $G$, a continuous $G$-action on a topological space $X$ is
called \textit{cohomologically equivariantly formal} if the Serre spectral sequence of the associated Borel
fibration $X \xrightarrow[]{} EG \times_{G} X \xrightarrow[]{} BG$ collapses at the $E_{2}$ page.\par
We simply call such actions and spaces \textit{equivariantly formal}.
\end{definition}
\begin{definition}
Given a ring $R$, the \textit{equivariant cohomology ring} with coefficients in $R$ is
$H^*_T(X;R)=H^*(X_T;R)$; it is a module over $H^*(BT;R)$.
\end{definition}

\subsection{GKM graph}
\begin{construction}\label{conFiltrationOrbitType}
For a $T$-action on a topological space $X$ consider the filtration
\begin{equation}\label{eqEquivFiltrGeneralX}
X_0\subset X_1\subset\cdots \subset X_k
\end{equation}
where $X_i$ is the union of all orbits of the action having dimension $\leqslant i$. In other words,
\[
X_i=\{x\in X\mid \dim T_x\geqslant k-i\}
\]
according to the natural homeomorphism $Tx\cong T^k/T_x$. Filtration~\eqref{eqEquivFiltrGeneralX} is called the orbit type filtration, and $X_i$ the equivariant $i$-skeleton of $X$. Each $X_i$ is $T$-stable.
\end{construction}

\begin{definition}
    A compact smooth $T$-manifold $X$ is called \textit{a manifold of GKM type} if $X^T$ is finite and $X_1$ is a finite union of $T$-invariant 2-spheres meeting only at fixed points.
\end{definition}

\begin{definition}
    A manifold of GKM type is called \textit{a GKM manifold} if it is equivariantly formal.
\end{definition}

\begin{remark}
    All $Q_{\Gamma,\lambda}$ are of GKM type, but not all of them are equivariantly formal, as will be demonstrated later on.
\end{remark}




For a manifold of GKM type, the invariant spheres in the one-skeleton
determine a finite graph. Its valence at a fixed point is the number of
tangent weights there, and hence one half of the real dimension of the
manifold.

\begin{definition}\label{graph_def}
A \textit{GKM graph} $\mathcal G$ of a manifold $X$ of GKM type is the finite
graph with vertex set $V=X^T=X_0$ and edges given by the invariant
2-spheres in $X_1$. Each edge is labeled by the tangent weight of the
corresponding sphere, determined up to sign. Thus the edge labels form
the \emph{axial function}
\[
\alpha\colon E_{\mathcal G}\longrightarrow
\operatorname{Hom}(T,S^1)/\{\pm1\}.
\]
\end{definition}

\section{Isospectral matrices}\label{sect3}
 In this section we define the main objects of this paper---the manifolds of isospectral matrices. The graph construction appearing in Theorem~\ref{Th} is given at the beginning of Section~\ref{Proof}.

\subsection{Skew-symmetric matrices}
\begin{construction}\label{def_q}
Let $Q_{2n}$ denote the real vector space of all real skew-symmetric
$2n\times2n$ matrices. It has dimension $n(2n-1)$. We regard such
matrices as $n\times n$ matrices of $2\times2$ blocks. A matrix
$A\in Q_{2n}$ is called $\Gamma$-shaped if its $(i,j)$-block vanishes
whenever $i\neq j$ and $\{i,j\}\notin E_\Gamma$; equivalently,
\[
a_{2i-1,2j-1}=a_{2i-1,2j}=a_{2i,2j-1}=a_{2i,2j}=0.
\]
The space $Q_\Gamma$ of all $\Gamma$-shaped matrices is a real vector
space of dimension $n+4|E_\Gamma|$.
\par

The eigenvalues of a real skew-symmetric matrix have the form
$\{\pm i\lambda_1,\ldots,\pm i\lambda_n\}$. For a prescribed set
$\lambda=\{\lambda_1,\ldots,\lambda_n\}$ of real numbers, let
$Q_{\Gamma,\lambda}\subset Q_\Gamma$ be the corresponding isospectral
set.

\end{construction}

\begin{definition}[The spectral map]
Put
\[
J=\begin{pmatrix}0&1\\-1&0\end{pmatrix}
\]
and let $Q_{2n}^{\mathrm{reg}}\subset Q_{2n}$ be the open subset of
matrices whose eigenvalues can be written as
$\{\pm i\mu_1,\ldots,\pm i\mu_n\}$ with
$\mu_i^2\neq\mu_j^2$ for $i\neq j$. Fix the open Weyl chamber of
type~$D_n$
\[
C_D^\circ=\{\mu\in\mathbb R^n\mid
\mu_1>\mu_2>\cdots>\mu_{n-1}>|\mu_n|\}.
\]
For $A\in Q_{2n}^{\mathrm{reg}}$, let $\chi(A)$ be the unique point
$\mu\in C_D^\circ$ for which $A$ is conjugate under $SO(2n)$ to
$\operatorname{diag}(\mu_1J,\ldots,\mu_nJ)$. This defines the spectral
map
\[
\chi\colon Q_{2n}^{\mathrm{reg}}\longrightarrow C_D^\circ.
\]
Set $Q_\Gamma^{\mathrm{reg}}=Q_\Gamma\cap Q_{2n}^{\mathrm{reg}}$ and
let
\[
\chi_\Gamma=\left.\chi\right|_{Q_\Gamma^{\mathrm{reg}}}
\colon Q_\Gamma^{\mathrm{reg}}\longrightarrow C_D^\circ.
\]
\end{definition}

\begin{definition}[Generic spectra]
Assume that $\lambda_i\neq0$ and
$\lambda_i^2\neq\lambda_j^2$ for $i\neq j$. Write the absolute values
in decreasing order as $a_1>\cdots>a_n>0$ and put
\[
\lambda^+=(a_1,\ldots,a_{n-1},a_n),\qquad
\lambda^-=(a_1,\ldots,a_{n-1},-a_n).
\]
The spectrum $\lambda$ is called \emph{generic} if both $\lambda^+$
and $\lambda^-$ are regular values of $\chi_\Gamma$.
\end{definition}

\begin{proposition}\label{smooth}
For a generic spectrum $\lambda$, the space $Q_{\Gamma,\lambda}$ is a
compact smooth manifold of real dimension $4|E_\Gamma|$.
\end{proposition}

\begin{proof}
The normal-form theorem for real skew-symmetric matrices gives the
stated block-diagonal representative. On the regular set, the type~$D_n$
Weyl group acts freely on the possible representatives, and every orbit
meets $C_D^\circ$ in exactly one point. The eigenvalue branches, and hence
$\chi$, vary smoothly on $Q_{2n}^{\mathrm{reg}}$. Therefore
$\chi_\Gamma$ is a smooth map obtained by restricting $\chi$ to the open
subset $Q_\Gamma^{\mathrm{reg}}$ of the linear subspace
$Q_\Gamma\subset Q_{2n}$.

The two chamber points $\lambda^+$ and $\lambda^-$ have the same
unordered spectrum and opposite Pfaffian signs. Consequently,
\[
Q_{\Gamma,\lambda}
=\chi_\Gamma^{-1}(\lambda^+)
\sqcup\chi_\Gamma^{-1}(\lambda^-).
\]
By the regular-value theorem, both level sets are smooth submanifolds of
codimension $n$ in the $(n+4|E_\Gamma|)$-dimensional real vector space
$Q_\Gamma$. Their disjoint union therefore has real dimension
$4|E_\Gamma|$.

Fixing the spectrum fixes
$-\operatorname{tr}(A^2)=2\sum_j a_j^2$, and hence the Euclidean norm of
$A$. The isospectral set is also closed, since the coefficients of the
characteristic polynomial are fixed. Thus $Q_{\Gamma,\lambda}$ is
compact. Finally, Sard's theorem shows that the critical values of
$\chi_\Gamma$ have measure zero. The reflection
$(\mu_1,\ldots,\mu_n)\mapsto(\mu_1,\ldots,-\mu_n)$ preserves
$C_D^\circ$, so the chamber points for which either $\lambda^+$ or
$\lambda^-$ is critical also form a measure-zero set. Thus the
genericity condition above holds outside a measure-zero subset.
\end{proof}

In what follows, $\lambda$ is always assumed to be generic in the sense
defined above. We next show that $Q_{\Gamma,\lambda}$ is disconnected
and decompose it into two diffeomorphic Pfaffian parts.
\begin{proposition}
The manifold $Q_{\G,\lambda}$ is the disjoint union of two nonempty
open-and-closed submanifolds $Q_{\G,\lambda}^{+}$ and
$Q_{\G,\lambda}^{-}$. These submanifolds are diffeomorphic.
\end{proposition}
\begin{proof}
For every skew-symmetric matrix $A$, one has
$\det(A)=\pf(A)^2$. The Pfaffian is nonzero on
$Q_{\Gamma,\lambda}$, so its sign defines two nonempty open-and-closed
parts, denoted by $Q_{\Gamma,\lambda}^{+}$ and
$Q_{\Gamma,\lambda}^{-}$.

Let $B=\operatorname{diag}(-1,1,\ldots,1)$ and define
$f(A)=B^TAB$. Conjugation by $B$ preserves the $\Gamma$-shape and the
spectrum. Moreover,
\[
\pf(B^TAB)=\det(B)\pf(A)=-\pf(A).
\]
Thus $f$ interchanges the two Pfaffian parts. Since $f$ is an
involution, it restricts to a diffeomorphism between them.
\end{proof}

\begin{construction}
The block-diagonal torus $T^n=SO(2)^n$ acts on
$Q_{\Gamma,\lambda}$ by conjugation. This action preserves the
$\Gamma$-shape, the spectrum, and the Pfaffian.
\end{construction}

\subsection{Hermitian matrices}
We also recall the analogous construction for Hermitian matrices from
\cite{A1}.

\begin{construction}\label{def_m}
    Let $M_n$ be the real vector space of Hermitian $n\times n$
    matrices. A matrix $A=(a_{i,j})$ is called $\Gamma$-shaped if
    $a_{i,j}=0$ whenever $i\ne j$ and $\{i,j\}\notin E_\Gamma$. The
    space of all Hermitian $\Gamma$-shaped matrices is denoted by
    $M_\Gamma$. For a prescribed spectrum
    $\lambda$, let
    $M_{\Gamma,\lambda}\subset M_\Gamma$ be the corresponding
    isospectral set.
\end{construction}

For generic $\lambda$, the space $M_{\Gamma,\lambda}$ is a smooth
manifold \cite[Construction~2.24]{A1}. The diagonal torus
\[
T^n=\{\operatorname{diag}(t_1,\ldots,t_n)\mid
t_i\in\mathbb C,\ |t_i|=1\}
\]
acts on it by conjugation, since this action preserves both the
$\Gamma$-shape and the spectrum.

\subsection{$\G = K_2$ case}

We first illustrate the two matrix constructions for the graph $K_2$.
\begin{example}\label{firstEx}
For $\G=K_2$, the space $M_\G$ consists of all Hermitian
$2\times2$ matrices. The fixed points of the conjugation action of
$T^2$ on $M_{\G,\lambda}$ are the two diagonal matrices with spectrum
$\lambda$. Moreover, $M_{\G,\lambda}\cong S^2$, and the whole
manifold is an invariant 2-sphere. Hence the underlying graph of its
GKM graph is $K_2$.
\end{example}

\begin{example}\label{secondEx}
Let $\G = K_2$. Then $Q_{\G,\lambda}$ is a manifold of $4 \times 4$ isospectral skew-symmetric matrices. The torus $T^2 = U(1) \times U(1) \cong SO(2) \times SO(2)$ acts on it by conjugation. 
\par
\begin{statement}
    The two Pfaffian parts satisfy $Q_{\Gamma,\lambda}^{+} \cong Q_{\Gamma,\lambda}^{-} \cong  S^2\times S^2$ and $T^2$ acts componentwise.
\end{statement}
For a regular spectrum, each Pfaffian part is a regular adjoint
orbit $SO(4)/T^2$. This homogeneous space is the oriented Grassmannian
$\operatorname{Gr}_2^+(\mathbb R^4)$, which is diffeomorphic to
$S^2\times S^2$; see \cite[Section~6]{Baird}. Under the standard
$SO(4)$-equivariant identification, the maximal torus acts on the two
sphere factors componentwise.

For this example choose $\lambda_1,\lambda_2>0$, which does not change
the skew-symmetric isospectral manifold. The fixed points of
$Q_{\Gamma,\lambda}^{+}$ are the following four matrices: 

$$\begin{pmatrix}
    \begin{bmatrix}
       0 & \lambda_1 \\
       -\lambda_1 & 0
    \end{bmatrix}   
    & 0 \\
    0 &
    \begin{bmatrix}
       0 & \lambda_2 \\
       -\lambda_2 & 0
    \end{bmatrix} 
\end{pmatrix}
,
\begin{pmatrix}
    \begin{bmatrix}
       0 & \lambda_2 \\
       -\lambda_2 & 0
    \end{bmatrix}   
    & 0 \\
    0 &
    \begin{bmatrix}
       0 & \lambda_1 \\
       -\lambda_1 & 0
    \end{bmatrix} 
\end{pmatrix} ,$$
$$
\begin{pmatrix}
    \begin{bmatrix}
       0 & -\lambda_1 \\
       \lambda_1 & 0
    \end{bmatrix}   
    & 0 \\
    0 &
    \begin{bmatrix}
       0 & -\lambda_2 \\
       \lambda_2 & 0
    \end{bmatrix} 
\end{pmatrix}
,
\begin{pmatrix}
    \begin{bmatrix}
       0 & -\lambda_2 \\
       \lambda_2 & 0
    \end{bmatrix}   
    & 0 \\
    0 &
    \begin{bmatrix}
       0 & -\lambda_1 \\
       \lambda_1 & 0
    \end{bmatrix} 
\end{pmatrix}.$$

There are also four fixed points of the $T^2$-action on the manifold $Q_{\Gamma, \lambda}^{-}$. They are obtained by changing the sign of the top-left $2\times2$
block in each displayed matrix.

\begin{remark}
    This particular manifold can also be obtained as the manifold of full flags of type $D_{2}$. Indeed, each component is $SO(4)/T^2$, where $T^2$ is a maximal torus of $SO(4)$.
\end{remark}

\begin{remark}
    The Weyl group $W(D_n)=S_n\ltimes\mathbb Z_2^{n-1}$ has
    four elements for $n=2$, corresponding to the four fixed points in
    each Pfaffian part.
\end{remark}

The one-skeleton is described explicitly in Section~\ref{Proof}. Its
invariant 2-spheres join fixed matrices obtained by interchanging
$\lambda_1$ and $\lambda_2$ while preserving the parity of the signs
above the diagonal. Thus two spheres pass through each fixed point of
$Q_{\Gamma,\lambda}^{+}$, and there are four such spheres in total.
The corresponding GKM graph is shown below.

\centering
\begin{tikzpicture} 
    \begin{scope} [vertex style/.style={draw,
                                       circle,
                                       minimum size=3mm,
                                       inner sep=3pt,
                                       outer sep=0pt,
                                       shade}]
    \path 
    (0 , 2.5) coordinate[vertex style] (c')
    (5 , 2.5) coordinate[vertex style] (d')
    (0 , 3.5) coordinate[vertex style] (e')
    (5 , 3.5) coordinate[vertex style] (f')
    
       ; 
    \end{scope}

     \begin{scope} [edge style/.style={draw=black}]
        \draw [edge style] (c')--(d');
        \draw [edge style] (c')--(f');
        \draw [edge style] (e')--(d');
        \draw [edge style] (e')--(f');
     \end{scope}
     \node[above left] at (e') {$\lambda_1, \lambda_2$};
     \node[below left] at (c') {$-\lambda_1, -\lambda_2$};
     \node[above right] at (f') {$\lambda_2, \lambda_1$};
     \node[below right] at (d') {$-\lambda_2, -\lambda_1$};

  \end{tikzpicture}

\par
{\small Horizontal edges: $\epsilon_1-\epsilon_2$; diagonal edges:
$\epsilon_1+\epsilon_2$ (up to sign).\par}

\end{example}

Examples~\ref{firstEx} and \ref{secondEx} illustrate the relation between
the GKM graphs of $M_{K_2,\lambda}$ and $Q_{K_2,\lambda}$.
Theorem~\ref{Th} extends this relation to every simple connected graph.

\begin{definition}
A fixed point of the Hermitian matrix manifold is denoted by $A_\sigma$,
where $\sigma\in S_n$. Fixed points of the skew-symmetric matrix
manifold are denoted by $A_{s,\sigma}$, where $s\in\mathbb Z_2^n$.
\end{definition}

In this notation, the vertex and edge correspondence asserted by
Theorem~\ref{Th} is given explicitly in Construction~\ref{graph}.

\section{Proof of Theorem~\ref{Th}}\label{Proof}
    \begin{construction}\label{graph}
    Let $\G=([n],E_\G)$ be a simple graph, and let $\mathcal G_M$ be
    the GKM graph of $M_{\G,\lambda}$. We construct a graph
    $\mathcal G_Q$ as follows. Let $\tau_{ij}s$ denote the vector
    obtained from $s\in\mathbb Z_2^n$ by interchanging its $i$-th and
    $j$-th coordinates, and let $\mathbf e_i$ be the $i$-th standard
    vector of $\mathbb Z_2^n$.

    For every vertex $A_\sigma$ of $\mathcal G_M$, introduce $2^n$
    vertices $A_{s,\sigma}$ of $\mathcal G_Q$. For every edge between
    $A_\sigma$ and $A_{\sigma\cdot(i,j)}$ in $\mathcal G_M$ and every
    $s\in\mathbb Z_2^n$, introduce the two edges
    \[
    A_{s,\sigma}A_{\tau_{ij}s,\,\sigma\cdot(i,j)}
    \quad\text{and}\quad
    A_{s,\sigma}A_{\tau_{ij}s+\mathbf e_i+\mathbf e_j,\,
    \sigma\cdot(i,j)}.
    \]
    Give these edges the respective axial weights
    \[
    \epsilon_i-\epsilon_j
    \quad\text{and}\quad
    \epsilon_i+\epsilon_j.
    \]
    Here $\epsilon_1,\ldots,\epsilon_n$ is the standard basis of
    $\operatorname{Hom}(T^n,S^1)\cong\mathbb Z^n$, so
    $\epsilon_i(t_1,\ldots,t_n)=t_i$. Thus every edge of
    $\mathcal G_M$ lifts to $2^{n+1}$ edges of the full graph
    $\mathcal G_Q$, with $2^n$ lifted edges in
    each of its two Pfaffian parts.
\end{construction}
   We next recall from \cite{A1} the facts about isospectral Hermitian
   matrices required for the proof.
    
\subsection{Hermitian matrices}\label{M_GKM}
The following statements describe the GKM graph of the $T^n$-action on
$M_{\G,\lambda}$.
\begin{proposition}[see {\cite[Remark 2.33]{A1}}]
    Independently of $\G$, the torus action on $M_{\G, \lambda}$ has exactly $n!$ isolated fixed points. Fixed points are given by the diagonal Hermitian matrices $A_\sigma = diag(\lambda_{\sigma(1)}, \ldots, \lambda_{\sigma(n)})$.
\end{proposition}

\begin{proposition}[see {\cite[Lemma 2.36]{A1}}]
    If $\sigma = \tau \cdot (i,j)$ with
    $\{i,j\}\in E_{\G}$, then the matrices $A_{\sigma}$ and
    $A_{\tau}$ are connected by a 2-sphere of isospectral matrices of the form
    $$
    \begin{pmatrix}
        \lambda_{\sigma(1)} & 0 & & \cdots & & 0\\
        0 & \ddots & &  &\\
        & & * & * &  & \vdots\\
        \vdots & & * & * & &\\
        & & & & \ddots & 0\\
        0 & & \cdots & & 0 & \lambda_{\sigma(n)}
    \end{pmatrix}
    \sim
    \begin{pmatrix}
        \lambda_{\tau(1)} & 0 & & \cdots & & 0\\
        0 & \ddots & &  &\\
        & & * & * &  & \vdots\\
        \vdots & & * & * & &\\
        & & & & \ddots & 0\\
        0 & & \cdots & & 0 & \lambda_{\tau(n)}
    \end{pmatrix}
    $$
    with the block on positions $(i,j)$. Here $\sim$ denotes unitary
    conjugacy within the principal block indexed by $i,j$.
\end{proposition}

These $T$-invariant 2-spheres give edges of the GKM graph. The tangent
representations below show that there are no further edges.

\begin{proposition}[see {\cite[Proof of Proposition 2.34]{A1}}]
    For $i<j$, let
    \[
    V_{ij}=\{A=(a_{i',j'})\in M_n\mid a_{i',j'}=0
    \text{ unless }(i',j')=(i,j)\text{ or }(j,i)\}.
    \]
    Then $V_{ij}$ is a $T$-invariant irreducible real
    two-dimensional representation of weight $\epsilon_i-\epsilon_j$,
    where $\epsilon_i(t_1,\ldots,t_n)=t_i$.
\end{proposition}

\begin{corollary}
    At any fixed point $A_\sigma$, the tangent representation $T_{A_{\sigma}}M_{\G, \lambda}$ has weights of the form
    $ \{\epsilon_i - \epsilon_j\,|\, \{i,j\} \in E_\G\}$.
     Moreover, $T_{A_{\sigma}}M_{\G, \lambda} = \bigoplus\limits_{\{i,j\}\in E_{\G}} V_{ij}$.
\end{corollary}

The dimension count shows that the preceding spheres account for all
tangent weights. We therefore obtain the following description.
\begin{proposition}\label{Her}
    The GKM graph of $M_{\G, \lambda}$ has $n!$ vertices $A_\sigma$,
    labeled by $\sigma\in S_n$. There is an edge between
    $A_{\sigma}$ and $A_{\tau}$ if and only if
    $\sigma = \tau \cdot (i,j)$ and $\{i,j\} \in E_\G$.
    The weight of this edge is $\epsilon_i-\epsilon_j$, up to sign.
\end{proposition}

\subsection{Skew-symmetric matrices}
We now determine the GKM graph of $Q_{\G,\lambda}$ and thereby prove
Theorem~\ref{Th}.

\begin{proposition}\label{need}
    Fixed points of $Q_{\G,\lambda}$ are block-diagonal matrices $A_{s,\sigma}$ for $\sigma \in S_n$ and $s \in \mathbb{Z}_2^n$, where
    $$
    A_{s,\sigma} =
    \begin{pmatrix}
    \begin{bmatrix}
    0        & (-1)^{s_1} \cdot \lambda_{\sigma(1)} \\
    -(-1)^{s_1} \cdot\lambda_{\sigma(1)}       & 0 
    \end{bmatrix}
    & \hdotsfor[1.5]{4} &
    \begin{bmatrix}
        0 &&&&&&&& 0 \\
        0 &&&&&&&& 0
    \end{bmatrix} \\
    \vdots & & & & \ddots &  \vdots \\
    \begin{bmatrix}
        0 &&&&&&&& 0 \\
        0 &&&&&&&& 0
    \end{bmatrix}
     & \hdotsfor[1.5]{4} &
     \begin{bmatrix}
    0        & (-1)^{s_n} \cdot\lambda_{\sigma(n)} \\
    -(-1)^{s_n} \cdot\lambda_{\sigma(n)}       & 0 
    \end{bmatrix}
    
\end{pmatrix}$$
\end{proposition}

\begin{proof}
    The torus acting on $Q_{\G,\lambda}$ consists of block-diagonal
    matrices of the form
  $$  
\begin{pmatrix}
    \begin{bmatrix}
    \cos\phi        & \sin\phi \\
    -\sin\phi       & \cos\phi
    \end{bmatrix}
    & \hdotsfor[1.5]{2} &
    \begin{bmatrix}
        0 & 0 \\
        0 & 0
    \end{bmatrix} \\
    \vdots &
    \begin{bmatrix}
     \cos\psi        & \sin\psi  \\
    -\sin\psi        & \cos\psi 
    \end{bmatrix}
             &   &  \vdots  \\
    \vdots &   & \ddots &  \vdots \\
    \begin{bmatrix}
        0 & 0 \\
        0 & 0
    \end{bmatrix}
     & \hdotsfor[1.5]{2} &
    \begin{bmatrix}
     \cos\theta        & \sin\theta   \\
    -\sin\theta        & \cos\theta 
    \end{bmatrix}
\end{pmatrix}
$$
    and acts by conjugation. Only its $(i,i)$- and $(j,j)$-blocks act
    on the $(i,j)$-block of a matrix in $Q_{\G,\lambda}$. It therefore
    suffices to consider the following $4\times4$ block calculation:

\[
\resizebox{\textwidth}{!}{$\displaystyle
 \begin{pmatrix}
    \begin{bmatrix}
    \cos\phi        & -\sin\phi \\
    \sin\phi       & \cos\phi
    \end{bmatrix}
    &
    \begin{bmatrix}
        0 & 0 \\
        0 & 0
    \end{bmatrix} \\
    \begin{bmatrix}
        0 & 0 \\
        0 & 0
    \end{bmatrix}
     & 
    \begin{bmatrix}
     \cos\theta        & -\sin\theta   \\
    \sin\theta        & \cos\theta 
    \end{bmatrix}
\end{pmatrix}
\cdot
\begin{pmatrix}
    \begin{bmatrix}
    *
    \end{bmatrix}
    &
    \begin{bmatrix}
        a & b \\
        c & d
    \end{bmatrix} \\
    \begin{bmatrix}
       *
    \end{bmatrix}
     & 
    \begin{bmatrix}
     *
    \end{bmatrix}
\end{pmatrix}
\cdot
\begin{pmatrix}
    \begin{bmatrix}
    \cos\phi        & \sin\phi \\
    -\sin\phi       & \cos\phi
    \end{bmatrix}
    &
    \begin{bmatrix}
        0 & 0 \\
        0 & 0
    \end{bmatrix} \\
    \begin{bmatrix}
        0 & 0 \\
        0 & 0
    \end{bmatrix}
     & 
    \begin{bmatrix}
     \cos\theta        & \sin\theta   \\
    -\sin\theta        & \cos\theta 
    \end{bmatrix}
\end{pmatrix}
$}
\]

Denote the resulting top-right block by $    \begin{bmatrix}
        e & f \\
        g & h
    \end{bmatrix} $. Direct multiplication gives
{\small
\begin{equation}\label{eq}
    \begin{cases}
    (a-d) \cdot \cos(\phi + \theta) + (a + d) \cdot \cos(\phi - \theta)  \\ \qquad-(b+c) \cdot \sin(\phi + \theta) +  (b-c)\cdot \sin(\phi - \theta) = 2e\\
    (a-d) \cdot \sin(\phi + \theta) - (a + d) \cdot \sin(\phi - \theta)  \\ \qquad+ (b+c)\cdot \cos(\phi + \theta) +  (b-c)\cdot \cos(\phi - \theta) = 2f\\
    (a-d) \cdot \sin(\phi + \theta) + (a + d) \cdot \sin(\phi - \theta)  \\ \qquad+ (b+c)\cdot \cos(\phi + \theta) -  (b-c)\cdot \cos(\phi - \theta) = 2g \\
    -(a-d) \cdot \cos(\phi + \theta) + (a + d) \cdot \cos(\phi - \theta)  \\ \qquad+ (b+c) \cdot \sin(\phi + \theta)+  (b-c)\cdot \sin(\phi - \theta) = 2h
    \end{cases}
\end{equation}
}

At a fixed point, $e=a$, $f=b$, $g=c$, and $h=d$. Substituting these
identities into \eqref{eq} gives
{\small
\begin{equation}\label{eq-fixed}
    \begin{cases}
    (a-d) \cdot \cos(\phi + \theta) + (a + d) \cdot \cos(\phi - \theta)  \\ \qquad-(b+c) \cdot \sin(\phi + \theta) +  (b-c)\cdot \sin(\phi - \theta) = 2a\\
    (a-d) \cdot \sin(\phi + \theta) - (a + d) \cdot \sin(\phi - \theta)  \\ \qquad+ (b+c)\cdot \cos(\phi + \theta) +  (b-c)\cdot \cos(\phi - \theta) = 2b\\
    (a-d) \cdot \sin(\phi + \theta) + (a + d) \cdot \sin(\phi - \theta)  \\ \qquad+ (b+c)\cdot \cos(\phi + \theta) -  (b-c)\cdot \cos(\phi - \theta) = 2c \\
    -(a-d) \cdot \cos(\phi + \theta) + (a + d) \cdot \cos(\phi - \theta)  \\ \qquad+ (b+c) \cdot \sin(\phi + \theta)+  (b-c)\cdot \sin(\phi - \theta) = 2d
    \end{cases}
\end{equation}
}
If the block is fixed for every $\phi,\theta$, system~\eqref{eq-fixed}
implies $a=b=c=d=0$.

Indeed, take $\phi=\pi$ and $\theta=0$. The block is then sent
to its negative, so a fixed block must be zero.
\par 
Thus every non-diagonal block of a fixed point is zero. Conversely,
each matrix $A_{s,\sigma}$ is fixed by direct calculation.

\end{proof}

To construct the GKM graph of $Q_{\G,\lambda}$, we determine its
invariant 2-spheres. For \mbox{$i<j\leq n$}, let $U_{ij}\subset Q_{2n}$ be
    $$ U_{ij} = \{A = (a_{k,l}) \, |\, a_{k,l} = 0 \; \text{unless}\; (k,l) \; \text{is in the block} \;(i,j) \; \text{or} \; (j,i)\}.
    $$
    Thus only $a_{2i-1,2j-1}$, $a_{2i-1,2j}$,
    $a_{2i,2j-1}$, $a_{2i,2j}$, and their skew-symmetric counterparts
    may be nonzero. There is a decomposition
    $U_{ij}=U^+_{ij}\oplus U^-_{ij}$ in which the $(i,j)$-blocks of
    elements of $U^+_{ij}$ and $U^-_{ij}$ have the respective forms
    $$B^+(u,v)=
    \begin{pmatrix}
        u & v\\
        v & -u
    \end{pmatrix},
    \qquad B^-(u,v)=
    \begin{pmatrix}
        u & v\\
        -v & u
    \end{pmatrix}
    $$
    The two subspaces are two-dimensional and have trivial
    intersection, so their sum is direct.
\begin{proposition}\label{weightspaces}
    The spaces $U^{+}_{ij}$ and $U^{-}_{ij}$ are $T$-invariant and
    irreducible over $\mathbb R$. Their respective weights are
    $\epsilon_i+\epsilon_j$ and $\epsilon_i-\epsilon_j$.
\end{proposition}

\begin{proof}
    For $U^-_{ij}$, substitute $a=d$ and $b=-c$ into
    \eqref{eq}. This gives
    
\begin{equation*}
\begin{cases}
    a \cdot \cos(\phi - \theta) + b\cdot \sin(\phi - \theta) = \alpha\\
     -a \cdot \sin(\phi - \theta) +  b\cdot \cos(\phi - \theta) = \gamma\\
    a \cdot \sin(\phi - \theta)  -  b\cdot \cos(\phi - \theta) = \delta \\
    a \cdot \cos(\phi - \theta) + b\cdot \sin(\phi - \theta) = \beta,
\end{cases}
\end{equation*}
for some parameters $\alpha, \gamma, \delta$ and $\beta$. Since $\alpha = \beta$ and $\gamma = -\delta$, $U^{-}_{ij}$ is a two-dimensional $T$-invariant representation. The weight of this representation is $\epsilon_i - \epsilon_j$, since

\begingroup\centering
\begin{equation*}
  \begin{bmatrix}
    \cos\phi        & -\sin\phi \\
    \sin\phi       & \cos\phi
    \end{bmatrix}
    \cdot
     \begin{bmatrix}
    a       & b \\
    -b       & a
    \end{bmatrix}
     \cdot
     \begin{bmatrix}
    \cos\theta       & \sin\theta \\
    -\sin\theta       & \cos\theta
    \end{bmatrix} =
\end{equation*}
\endgroup
\begingroup\centering
\begin{equation*}
    =\begin{bmatrix}
    a \cdot \cos(\phi - \theta) + b\cdot \sin(\phi - \theta) &
    -a \cdot \sin(\phi - \theta) +  b\cdot \cos(\phi - \theta)\\
    a \cdot \sin(\phi - \theta)  -  b\cdot \cos(\phi - \theta)  &
    a \cdot \cos(\phi - \theta) + b\cdot \sin(\phi - \theta)
    \end{bmatrix}
\end{equation*}
\endgroup

This representation is irreducible because the angle $\phi-\theta$
ranges over all rotations, and no real line is invariant under all
rotations. The case of $U^+_{ij}$ is analogous, with angle $\phi+\theta$.
\end{proof}

\begin{proposition}
    At any fixed point $A_{s,\sigma}$ the tangent representation is
    $$T_{A_{s,\sigma}}Q_{\G, \lambda} = \bigoplus\limits_{\{i,j\}\in E_{\G}}\left(U^+_{ij}\oplus U^-_{ij}\right),$$
    with tangent weights of the form $ \{\epsilon_i + \epsilon_j, \epsilon_i - \epsilon_j\,|\, \{i,j\} \in E_\G\}$.
\end{proposition}
\begin{proof}
For each edge of the graph $\G$, the decomposition above gives the two
irreducible two-dimensional tangent representations $U^+_{ij}$ and
$U^-_{ij}$, whose weights were computed in
Proposition~\ref{weightspaces}.
Since $\chi_\Gamma$ is $T^n$-invariant and $A_{s,\sigma}$ is fixed,
its differential vanishes on the nontrivial representations
$U^\pm_{ij}$ of Proposition~\ref{weightspaces}; hence these subspaces
lie in $T_{A_{s,\sigma}}Q_{\Gamma,\lambda}$.
Since $\dim Q_{\G,\lambda}=4|E_\G|$, these representations account for
the entire tangent space.
\end{proof}

\subsection{Invariant spheres}
Fix $A_{s,\sigma}$ and an edge $\{i,j\}\in E_\G$, and put
\[
a_k=(-1)^{s_k}\lambda_{\sigma(k)}.
\]
Keep all blocks of $A_{s,\sigma}$ fixed except the principal
$4\times4$ block indexed by $i,j$.

Let $\mathcal S^-_{ij}(s,\sigma)$ be the set of matrices for which the
indicated principal block varies as
\begin{equation}\label{first}
\begin{pmatrix}
xJ & B^-(u,v)\\
-(B^-(u,v))^T & (a_i+a_j-x)J
\end{pmatrix},
\qquad
B^-(u,v)=\begin{pmatrix}u&v\\-v&u\end{pmatrix},
\end{equation}
where
\begin{equation}\label{sphere-minus}
u^2+v^2+
\left(x-\frac{a_i+a_j}{2}\right)^2
=\left(\frac{a_i-a_j}{2}\right)^2.
\end{equation}
Similarly, let $\mathcal S^+_{ij}(s,\sigma)$ be the set of matrices for
which the indicated principal block varies as
\begin{equation}\label{second}
\begin{pmatrix}
xJ & B^+(u,v)\\
-(B^+(u,v))^T & (x-a_i+a_j)J
\end{pmatrix},
\qquad
B^+(u,v)=\begin{pmatrix}u&v\\v&-u\end{pmatrix},
\end{equation}
where
\begin{equation}\label{sphere-plus}
u^2+v^2+
\left(x-\frac{a_i-a_j}{2}\right)^2
=\left(\frac{a_i+a_j}{2}\right)^2.
\end{equation}

\begin{theorem}\label{spheres}
The sets $\mathcal S^-_{ij}(s,\sigma)$ and
$\mathcal S^+_{ij}(s,\sigma)$ we have constructed above are $T^n$-invariant
2-spheres. The endpoints of $\mathcal S^-_{ij}(s,\sigma)$ are
\[
A_{s,\sigma}
\quad\text{and}\quad
A_{\tau_{ij}s,\,\sigma\cdot(i,j)},
\]
and its tangent weight is $\epsilon_i-\epsilon_j$. The endpoints of
$\mathcal S^+_{ij}(s,\sigma)$ are
\[
A_{s,\sigma}
\quad\text{and}\quad
A_{\tau_{ij}s+\mathbf e_i+\mathbf e_j,\,\sigma\cdot(i,j)},
\]
and its tangent weight is $\epsilon_i+\epsilon_j$.
\end{theorem}

\begin{proof}
The off-diagonal blocks in \eqref{first} and \eqref{second} belong,
respectively, to the representations $U^-_{ij}$ and $U^+_{ij}$.
Proposition~\ref{weightspaces} shows that the torus rotates the
coordinates $(u,v)$ through the angles $\phi_i-\phi_j$ and
$\phi_i+\phi_j$. Hence both sets are $T^n$-invariant and have the
claimed tangent weights.

It remains to check the spectrum and the endpoints. The characteristic
polynomial of a real skew-symmetric $4\times4$ matrix is
\[
t^4+c_2t^2+\pf^2.
\]
For the block in \eqref{first}, with $r^2=u^2+v^2$ and
$y=a_i+a_j-x$, one has
\[
c_2=x^2+y^2+2r^2,
\qquad
\pf=xy-r^2.
\]
Equation~\eqref{sphere-minus} is equivalent to
$xy-r^2=a_i a_j$ and consequently gives
$c_2=a_i^2+a_j^2$. Thus the characteristic polynomial is
\[
(t^2+a_i^2)(t^2+a_j^2).
\]
When $u=v=0$, the two solutions are
$(x,y)=(a_i,a_j)$ and $(a_j,a_i)$, which give precisely
$A_{s,\sigma}$ and $A_{\tau_{ij}s,\,\sigma\cdot(i,j)}$.

For the block in \eqref{second}, put $y=x-a_i+a_j$. Then
\[
c_2=x^2+y^2+2r^2,
\qquad
\pf=xy+r^2.
\]
Equation~\eqref{sphere-plus} is equivalent to
$xy+r^2=a_i a_j$, and again $c_2=a_i^2+a_j^2$. Its two fixed
points are obtained from
$(x,y)=(a_i,a_j)$ and $(-a_j,-a_i)$, yielding
$A_{s,\sigma}$ and
$A_{\tau_{ij}s+\mathbf e_i+\mathbf e_j,\,\sigma\cdot(i,j)}$.
The genericity condition $a_i^2\neq a_j^2$ makes both sphere radii
nonzero.
\end{proof}

The one-skeleton is precisely the union of these invariant spheres.
Indeed, the ambient nonzero weights $\epsilon_i\pm\epsilon_j$ are
pairwise nonproportional. A matrix with a one-dimensional torus orbit can therefore
have a nonzero coordinate in only one of these weight planes. All
remaining off-diagonal coordinates vanish. Fixing its spectrum then
gives one of the spheres in Theorem~\ref{spheres}, with the other
diagonal blocks fixed. Thus exactly $2|E_\G|$ invariant spheres meet
each fixed point. The formulas~\eqref{first}--\eqref{sphere-plus}
also show that distinct spheres in this list intersect only at fixed
points. This proves the following description.

\begin{proposition}\label{last}
The GKM graph of $Q_{\G,\lambda}$ has $2^n n!$ vertices
$A_{s,\sigma}$, where $\sigma\in S_n$ and $s\in\mathbb Z_2^n$.
For every $\{i,j\}\in E_\G$ it has exactly the following two edges
issuing from $A_{s,\sigma}$:
\[
A_{s,\sigma}A_{\tau_{ij}s,\,\sigma\cdot(i,j)}
\quad\text{of weight }\epsilon_i-\epsilon_j,
\]
and
\[
A_{s,\sigma}
A_{\tau_{ij}s+\mathbf e_i+\mathbf e_j,\,\sigma\cdot(i,j)}
\quad\text{of weight }\epsilon_i+\epsilon_j.
\]
Consequently the graph is $2|E_\G|$-valent. Each edge of the GKM graph
of $M_{\G,\lambda}$ lifts to $2^{n+1}$ edges of the full graph of
$Q_{\G,\lambda}$, or to $2^n$ edges in each Pfaffian part.
\end{proposition}

Propositions~\ref{Her} and \ref{last} establish Theorem~\ref{Th}.

\begin{corollary}
    The GKM graph $\mathcal G_Q^+$ of $Q_{\G,\lambda}^{+}$ is the
    subgraph of $\mathcal G_Q$ on the vertices satisfying
    \[
    \pf(A_{s,\sigma})=(-1)^{s_1+\cdots+s_n}\prod_{k=1}^n\lambda_k>0.
    \]
    The graph $\mathcal G_Q^-$ of $Q_{\G,\lambda}^{-}$ consists of the remaining
    vertices and their edges. If $\prod_k\lambda_k>0$, these are the
    even and odd sign parities, respectively; if the product is
    negative, the roles are reversed. The graphs $\mathcal G_Q^+$ and
    $\mathcal G_Q^-$ are the two connected components of
    $\mathcal G_Q$ and are isomorphic.
\end{corollary}
\begin{proof}
The first kind of edge exchanges two signed entries, while
following an edge of each kind for the same pair changes both signs
without changing the permutation. Since $\Gamma$ is connected, its
edge transpositions generate all permutations. Conjugating the
two-sign changes by these permutations generates every even sign
change. Thus each parity class is connected. Negating one coordinate identifies the two
underlying graphs (and reflects the corresponding axial coordinate).
\end{proof}

\section{Equivariant formality}\label{sect5}

We compare equivariant formality of $M_{\Gamma,\lambda}$ and
$Q_{\Gamma,\lambda}$ for generic $\lambda$. We first prove the
implication in Theorem~\ref{main} and then construct a counterexample
to its converse.
Cohomology is taken with integer coefficients unless $\mathbb Q$ is
indicated explicitly.
\begin{proposition}\label{almostTh}
If $Q_{\Gamma,\lambda}$ is equivariantly formal, then
$M_{\Gamma,\lambda}$ is equivariantly formal.
\end{proposition}

We use the following standard fixed-point result
\cite[Lemma~2.2]{MasPan}.

\begin{lemma}\label{Pan}
Let a torus $T$ act smoothly on a closed manifold $X$, suppose that
$X^T$ is finite, and let $H\subset T$ be a subtorus. If $X$ is
equivariantly formal over $\mathbb Q$, then each connected component
$Y$ of $X^H$ has $H^{\mathrm{odd}}(Y;\mathbb Q)=0$ and contains a
$T$-fixed point.
\end{lemma}

\begin{construction}
Multiplication by $i$ identifies $M_{\Gamma,\lambda}$ with a manifold
$M'_{\Gamma,\lambda}$ of skew-Her\-mi\-tian matrices. This identification
preserves the zero pattern and is equivariant for the diagonal torus.
Realification then gives a manifold
$M'_{\Gamma,\lambda,\mathbb R}$ of $2n\times2n$ real
skew-symmetric matrices. Under this map, a complex entry $a+bi$ is
represented by the real block
$\left(\begin{smallmatrix}
    a & -b\\
    b & a
\end{smallmatrix}\right)$
and the eigenvalues $\lambda_j$ become the pairs $\pm i\lambda_j$.
\end{construction}

\begin{example}
$$
\begin{pmatrix}
    k & l_1 + i\cdot l_2 & m_1 + i\cdot m_2 \\ 
    l_1 - i\cdot l_2 & r & s_1 + i\cdot s_2 \\
    m_1 - i\cdot m_2 & s_1 - i\cdot s_2 & q
\end{pmatrix}
\xrightarrow[]{M_{\Gamma, \lambda}\mapsto M'_{\Gamma, \lambda}} \\
$$

$$
\xrightarrow[]{}
\begin{pmatrix}
    i \cdot k & i \cdot l_1 - l_2 & i \cdot m_1 - m_2 \\ 
    i\cdot l_1 + l_2 & i \cdot r & i \cdot s_1 - s_2 \\
    i \cdot m_1 + m_2 & i \cdot s_1 + s_2 & i \cdot q
\end{pmatrix}
\xrightarrow[]{M'_{\Gamma, \lambda} \mapsto M'_{\Gamma, \lambda, \mathbb{R}}}
$$

$$
\xrightarrow[]{}
\begin{pmatrix}
\begin{bmatrix}
0 & -k\\
k & 0
\end{bmatrix} 
& 
\begin{bmatrix}
-l_2 & -l_1\\
l_1 & -l_2
\end{bmatrix}
& 
\begin{bmatrix}
-m_2 & -m_1\\
m_1 & -m_2
\end{bmatrix}
\\
\begin{bmatrix}
l_2 & -l_1\\
l_1 & l_2
\end{bmatrix}
& 
\begin{bmatrix}
0 & -r\\
r & 0
\end{bmatrix} 
& 
\begin{bmatrix}
-s_2 & -s_1\\
s_1 & -s_2
\end{bmatrix}
\\
\begin{bmatrix}
m_2 & -m_1\\
m_1 & m_2
\end{bmatrix}
& 
\begin{bmatrix}
s_2 & -s_1\\
s_1 & s_2
\end{bmatrix}
&
\begin{bmatrix}
0 & -q\\
q & 0
\end{bmatrix}
\end{pmatrix}
$$

\end{example}

To apply Lemma~\ref{Pan}, we describe the fixed point set of the
diagonal circle. It contains realifications of the Hermitian
isospectral manifolds for all choices of signs of the spectrum.
\begin{proposition}\label{my}
Let $S^1_\Delta=\{(t,\ldots,t)\}\subset T^n$ be the diagonal circle,
and write $b_j=|\lambda_j|$. Then
\[
(Q_{\Gamma,\lambda})^{S^1_\Delta}
=\bigsqcup_{\varepsilon\in\{\pm1\}^n}
 M'_{\Gamma,\varepsilon b,\mathbb R},
\qquad
\varepsilon b=(\varepsilon_1b_1,\ldots,\varepsilon_nb_n).
\]
Each summand is open and closed in the fixed point set.
\end{proposition}
\begin{proof}
A block fixed by the diagonal circle commutes with every planar
rotation. Setting $\phi=\theta$ and then $\phi=\pi/2$ in
\eqref{eq-fixed} gives $a=d$ and $b=-c$.
Conversely, every block
$\left(\begin{smallmatrix}a&b\\-b&a\end{smallmatrix}\right)$
is fixed by this circle. A real skew-symmetric matrix with these
blocks is exactly the realification of a skew-Hermitian matrix.

The spectrum of the corresponding Hermitian matrix must contain
one of $b_j,-b_j$ for each $j$. It need not equal the single
spectrum $\lambda$. Since the $b_j$ are nonzero and pairwise
distinct, the $2^n$ signed spectra are different. This proves the
disjoint union. Every summand is compact, and there are finitely
many summands, so each is open and closed in the fixed point set.
\end{proof}

\begin{proof}[Proof of Proposition~\ref{almostTh}]
Suppose that $Q_{\Gamma,\lambda}$ is equivariantly formal. After
extending coefficients to $\mathbb Q$, Proposition~\ref{my} identifies
$M'_{\Gamma,\lambda,\mathbb R}$ with an open-and-closed subset
of its diagonal-circle fixed point set. Then, by applying Lemma~\ref{Pan} componentwise, 
we have $H^{\mathrm{odd}}(M_{\Gamma,\lambda};\mathbb Q)=0$.
The Hermitian torus action is equivariantly formal over $\mathbb Q$
by parity.

Since $\lambda$ is simple, \cite[Theorem~2]{A2} applies: rational equivariant
formality of $M_{\Gamma,\lambda}$ implies that $\Gamma$ is an
indifference graph. For such graphs the Hermitian manifold has
free integral cohomology in even degrees
\cite[Corollary~4.6]{A-B}; hence its integral Borel spectral sequence
collapses. This proves the proposition.
\end{proof}

The comparison of the two GKM graphs suggests the following
natural converse, which is false as the counterexample below shows.
\begin{hypothesis}[Disproved]\label{converseHyp}
If $M_{\Gamma,\lambda}$ is equivariantly formal, then
$Q_{\Gamma,\lambda}$ is also equivariantly formal.
\end{hypothesis}
The following counterexample disproves this hypothesis.

\begin{proposition}\label{P3counter}
Let $P_3$ be the path with edges $\{1,2\}$ and $\{2,3\}$. For generic $\lambda$,
the manifold $M_{P_3,\lambda}$ is equivariantly formal,
while $Q_{P_3,\lambda}$ is not.
\end{proposition}
\begin{proof}
Note that the path $P_3$ is the unit interval graph represented by
$[0,1]$, $[3/4,7/4]$, and $[3/2,5/2]$. Its Hermitian matrices
are tridiagonal, with Hessenberg function $h=(2,3,3)$.
They have free integral cohomology in even degrees
\cite[Corollary~4.6]{A-B}.
Thus $M_{P_3,\lambda}$ is equivariantly formal.

We identify a Pfaffian part of $Q_{P_3,\lambda}$ with the
isospectral Hermitian space for the cycle $C_4$, whose edges
are $\{1,2\}$, $\{2,3\}$, $\{3,4\}$, and $\{1,4\}$. The action of $SU(4)$ on the real form of
$\Lambda^2\mathbb C^4$ fixed by the antilinear Hodge star gives
a surjection $\rho:SU(4)\to SO(6)$ with kernel $\{\pm I\}$,
and a Lie algebra isomorphism $d\rho$ \cite[Example~5.9]{KramerStroppel}. Choose the three real
coordinate planes with complexifications
\[
\langle e_{12},e_{34}\rangle,\qquad
\langle e_{13},e_{24}\rangle,\qquad
\langle e_{14},e_{23}\rangle,
\]
where $e_{ij}=e_i\wedge e_j$.
The weights on these three pairs of lines are respectively
$\pm(\mu_1+\mu_2)$, $\pm(\mu_1+\mu_3)$, and $\pm(\mu_1+\mu_4)$.
Put $J=\bigl(\begin{smallmatrix}0&1\\-1&0\end{smallmatrix}\bigr)$.
Choose the orientations of these three real coordinate planes so that, for $\sum_i\mu_i=0$,
$i\operatorname{diag}(\mu_1,\ldots,\mu_4)$ maps to
$\operatorname{diag}(a_1J,a_2J,a_3J)$, where
\[
a_1=\mu_1+\mu_2,\qquad
a_2=\mu_1+\mu_3,\qquad
a_3=\mu_1+\mu_4.
\]
Equivalently,
\begin{equation}\label{P3spectrum}
\mu=\tfrac12(a_1+a_2+a_3,\ a_1-a_2-a_3,\
-a_1+a_2-a_3,\ -a_1-a_2+a_3).
\end{equation}
The forbidden block $(1,3)$ in $Q_{P_3}$ consists of the real root
planes with weights $\epsilon_1+\epsilon_3$ and
$\epsilon_1-\epsilon_3$. These correspond to the Hermitian positions
$(1,3)$ and $(2,4)$, respectively, since
\[
\mu_1-\mu_3=a_1+a_3,\qquad
\mu_2-\mu_4=a_1-a_3.
\]
Consequently $H\mapsto d\rho(iH)$ identifies the trace-zero
Hermitian subspace supported on $C_4$ with $Q_{P_3}$.

Put $a_j=|\lambda_j|$ and define $\mu$ by \eqref{P3spectrum}.
The full isospectral set of real skew-symmetric matrices has two
$SO(6)$-orbits, distinguished by the Pfaffian sign.
Surjectivity of $\rho$ and equivariance of $d\rho$ identify the
orbit of $i\operatorname{diag}(\mu)$ with the positive-Pfaffian
orbit. Intersecting with the subspaces above gives
\[
Q_{P_3,\lambda}^{+}\cong M_{C_4,\mu}.
\]
In dimension six $\pf(-A)=-\pf(A)$, so the other part is
\[
Q_{P_3,\lambda}^{-}\cong M_{C_4,-\mu}.
\]
All differences $\mu_i-\mu_j$ are of the form $a_r\pm a_s$
up to sign. They are nonzero by our assumptions, so the two
Hermitian spectra are simple.

By \cite[Corollary~9.4]{Tridiag}, the space $M_{C_4,\mu}$
has nonzero odd rational cohomology for every simple spectrum.
Thus $H^{\mathrm{odd}}(Q_{P_3,\lambda};\mathbb Q)\ne0$.
Its $T^3$-fixed point set is finite by Proposition~\ref{need}.
If the action were equivariantly formal over $\mathbb Q$,
Lemma~\ref{Pan}, applied componentwise with the trivial subtorus
$H=\{e\}$, would give
$H^{\mathrm{odd}}(Q_{P_3,\lambda};\mathbb Q)=0$, a contradiction.
Hence the action is not equivariantly formal over $\mathbb Q$,
and therefore not with integer coefficients either.
\end{proof}

Propositions~\ref{almostTh} and \ref{P3counter} prove
Theorem~\ref{main}.

\begin{remark}
For generic $\lambda$, equivariant formality of
$M_{\Gamma,\lambda}$ is characterized by unit interval graphs
(indifference graphs) \cite{A2}. Proposition~\ref{almostTh}
makes this a necessary condition for equivariant formality of
$Q_{\Gamma,\lambda}$.
The path $P_3$ shows that it is not sufficient.
\end{remark}


{\small
\begin{thebibliography}{10}

\bibitem{Tridiag}
A.~Ayzenberg.
\newblock Space of isospectral periodic tridiagonal matrices.
\newblock {\em Algebr. Geom. Topol.}, 20(6):2957--2994, 2020.

\bibitem{Arrow}
A.~Ayzenberg and V.~Buchstaber.
\newblock Manifolds of isospectral arrow matrices.
\newblock {\em Mat. Sb.}, 212(5):3--36, 2021.

\bibitem{A-B}
A.~Ayzenberg and V.~Buchstaber.
\newblock Manifolds of isospectral matrices and {Hessenberg} varieties.
\newblock {\em Int. Math. Res. Not. IMRN}, 2021(21):16671--16692, 2021.

\bibitem{A1}
A.~Ayzenberg and V.~Buchstaber.
\newblock Cluster-permutohedra and submanifolds of flag varieties with torus actions.
\newblock {\em Int. Math. Res. Not. IMRN}, 2024(3):1931--1967, 2024.

\bibitem{Ay1}
A.~Ayzenberg, M.~Masuda, and G.~Solomadin.
\newblock How is a graph not like a manifold?
\newblock {\em Sb. Math.}, 214(6):793--815, 2023.

\bibitem{A2}
A.~Ayzenberg and K.~Sorokin.
\newblock Topological approach to diagonalization algorithms, 2022.
\newblock Preprint, arXiv:2204.06111.

\bibitem{Baird}
P.~Baird.
\newblock Harmonic morphisms onto {Riemann} surfaces and generalized analytic functions.
\newblock {\em Ann. Inst. Fourier (Grenoble)}, 37(1):135--173, 1987.

\bibitem{BB}
A.~Bia{\l}ynicki-Birula.
\newblock Some theorems on actions of algebraic groups.
\newblock {\em Ann. of Math. (2)}, 98(3):480--497, 1973.

\bibitem{KramerStroppel}
L.~Kramer and M.~J. Stroppel.
\newblock Hodge operators and exceptional isomorphisms between unitary groups.
\newblock {\em J. Lie Theory}, 33(1):329--360, 2023.

\bibitem{MasPan}
M.~Masuda and T.~Panov.
\newblock On the cohomology of torus manifolds.
\newblock {\em Osaka J. Math.}, 43(3):711--746, 2006.

\end{thebibliography}

}

\end{document}